\documentclass[12pt]{amsart}
\usepackage[utf8]{inputenc}
\usepackage[T1]{fontenc}
\usepackage{amsfonts}
\usepackage[leqno]{amsmath}
\usepackage{amssymb, comment, float}
\usepackage[dvipsnames]{xcolor}
\usepackage[foot]{amsaddr}
\usepackage[shortlabels]{enumitem}
\usepackage{mathdots}
\usepackage{thm-restate}
\usepackage[footskip=0.5in,  headheight = 0.5in, top=1.25in, bottom=1.25in,  right=1in,  left=1in]{geometry}
\usepackage{hyperref}
\hypersetup{
colorlinks,
linkcolor=black,
urlcolor=Blue,
citecolor=black,}
\usepackage[center]{caption}
\usepackage{eufrak}
\usepackage{marvosym}     
\usepackage{latexsym}
\usepackage[nodayofweek]{datetime}
\usepackage{tikz}
\usepackage{subfig}
\usepackage{thm-restate}

\newtheorem{theorem}{Theorem}[section]
\newtheorem{lemma}[theorem]{Lemma}

\newtheorem{question}[theorem]{Question}
\newtheorem{observation}[theorem]{Observation}

 \DeclareMathOperator{\dist}{{\text{\sf{dist}}}}

\def\dd{\hbox{-}}   

\newcommand{\poi}{\mathbb{N}}
\newcommand{\ol}{\overline}

\newcommand{\mca}{\mathcal}
\newcommand{\is}{\text{-induced-saturated}}

\newcounter{tbox}
\newcommand{\sta}[1]{\medskip\medskip\refstepcounter{tbox}\noindent{\parbox{\textwidth}{(\thetbox) \emph{#1}}}\vspace*{0.3cm}}
\newcommand{\mylongtitle}[1]{%
  \ifodd\value{page}%
    \protect\parbox{0.97\linewidth}{#1}\hfill%
  \else%
    \hfill\protect\parbox{0.97\linewidth}{#1}%
  \fi%
}

\title[Halfway to induced saturation for even cycles]{Halfway to induced saturation for even cycles}

\author{Xinyue Fan$^{\dagger}$}

\author{Sahab Hajebi$^{\dagger}$}

\author{Sepehr Hajebi$^{\dagger}$}

\author{Sophie Spirkl$^{\dagger, \ddagger}$}

\thanks{$^{\dagger}$ Department of Combinatorics and Optimization, University of Waterloo, Waterloo, Ontario, Canada}
\thanks{$^{\ddagger}$ We acknowledge the support of the Natural Sciences and Engineering Research Council of Canada (NSERC), [funding reference number RGPIN-2020-03912].
Cette recherche a \'et\'e financ\'ee par le Conseil de recherches en sciences naturelles et en g\'enie du Canada (CRSNG), [num\'ero de r\'ef\'erence RGPIN-2020-03912]. This project was funded in part by the Government of Ontario. This research was conducted while Spirkl was an Alfred P. Sloan Fellow.}
\date {\today}

\date{\today}

\begin{document}

\maketitle


\begin{abstract}
For graphs $G$ and $H$, we say that $G$ is \textit{$H$-free} if no induced subgraph of $G$ is isomorphic to $H$, and that $G$ is \textit{$H\is$} if $G$ is $H$-free but removing or adding any edge in $G$ creates an induced copy of $H$. A full characterization of graphs $H$ for which $H\is$ graphs exist remains elusive. Even the case where $H$ is a path -- now settled by the collective results of Martin and Smith, Bonamy et al., and Dvo\'{r}\v{a}k -- was already quite challenging. 

What if $H$ is a cycle? The complete answer for odd cycles was given by Behren et al., leaving the case of even cycles (except for the $4$-cycle) wide open. Our main result is the first step toward closing this gap: We prove that for every even cycle $H$, there is a graph $G$ with at least one edge such that $G$ is $H$-free but removing any edge from $G$ creates an induced copy of $H$ (in fact, we construct $H\is$ graphs for every even cycle $H$ on at most 10 vertices).
\end{abstract}

\section{Introduction}
The set of all positive integers is denoted by $\poi$, and for every integer $k$, the set of all positive integers no greater than $k$ is denoted by $\poi_k$ (so we have $\poi_k=\varnothing$ if and only if $k\leq 0$). Graphs in this paper have finite vertex sets, and the edges of a graph are regarded as $2$-subsets of its vertex set; in particular, loops and parallel edges are not allowed. Let $G=(V(G), E(G))$ be a graph. The complement of $G$ is denoted by $\ol{G}$. For every $e\in E(G)$, we denote by $G-e$ the graph obtained from $G$ by removing the edge $e$, and for every $e\in E(\ol{G})$, we denote by $G+e$ the graph obtained from $G$ by adding the edge $e$. 

For graphs $G$ and $H$, we say that $G$ is \textit{$H$-free} if no induced subgraph of $G$ is isomorphic to $H$. We also say that $G$ is \textit{$H$-induced-saturated} if $G$ is $H$-free, $G-e$ has an induced subgraph isomorphic to $H$ for every $e\in E(G)$, and $G+e$ has an induced subgraph isomorphic to $H$ for every $e\in E(\ol{G})$. This notion originated in a 2012 work by Martin and Smith \cite{martin}, and has since evolved around the study of graphs $H$ for which $H\is$ graphs exist. For instance, if $|V(H)|\neq 2$ and one of $H$ or $\ol{H}$ is complete, then $H$-induced-saturated graphs do not exist. Could the converse also be true? Martin and Smith \cite{martin} answered this in the negative (for $t\in \poi$, we write $P_t$ to denote the $t$-vertex path):

\begin{theorem}[Martin and Smith \cite{martin}]\label{thm:p4}
    There is no $P_4\is$ graph.
\end{theorem}
On the other hand, for various choices of $H$, the existence of $H\is$ graphs were proved in subsequent results by Behrens et al. \cite{behrens} and by Axenovich and Csik\'{o}s \cite{axenovich}. The current state of the art, however, is far from a characterization of all graphs $H$ for which $H\is$ graphs exist. 

Motivated by Theorem~\ref{thm:p4}, Axenovich and Csik\'{o}s \cite{axenovich} highlighted the case where $H$ is a path as one of particular interest. This is immediate for small paths: $P_1\is$ graphs do not exist, edgeless graphs on two or more vertices are $P_2\is$, complete graphs on three or more vertices are $P_3\is$, and by Theorem~\ref{thm:p4}, there is no $P_4\is$ graph. Moreover, we observed that the icosahedron is $\ol{P_5}\is$ (see Figure~\ref{fig:smallt}).

\begin{observation}\label{obs:icosa}
    The complement of the icosahedron is $P_5\is$.
\end{observation}
(We believe that, through a computer search for $P_5\is$ graphs, Bonamy et al. \cite{bonamy} have also come across the same graph and apparently did not recognize it as the complement of the icosahedron.)
\begin{figure}[t!]
    \centering
    \includegraphics[scale=0.6]{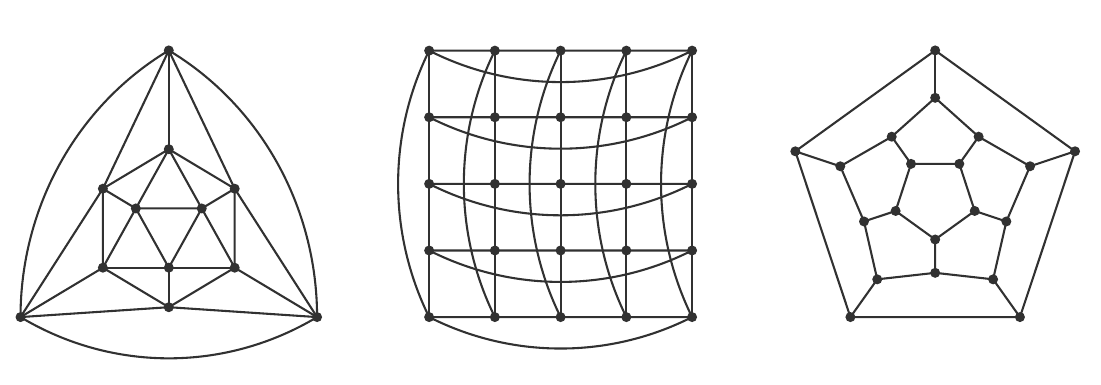}
    \caption{The icosahedron (left), the Cartesian product of two $5$-cycles  (middle) and the dodecahedron (right).}
    \label{fig:smallt}
\end{figure}
\medskip

Progress on paths with six or more vertices was made much more gradually. R\"{a}ty \cite{raty} observed that the (triangle-free) Clebsch graph is $P_6\is$. Cho et al. \cite{cho} constructed $P_{3s}\is$ graphs for all integers $s\geq 2$. Eventually, Dvo\v{r}\'{a}k \cite{ptdvorak} came up with a beautiful construction of $P_t\is$ graphs for all $t\geq 6$, which we describe below.

For every $t\in \poi$, let $VD_t$ be the graph with $2t$ vertices $\{u_i,v_i:\in \poi_t\}$ such that for all distinct $i,j\in \poi_t$, we have:
\begin{itemize}
    \item $u_i,u_j$ are adjacent if and only if $|i-j|\in \{1,t-1\}$;
    \item $v_i,v_j$ are adjacent if and only if $|i-j|\notin \{1,t-1\}$; and
    \item $u_i,v_j$ are adjacent if and only if $i=j$.
\end{itemize}
For instance, $VD_5$ is isomorphic to the Peterson graph.
The following was proved in \cite{ptdvorak}, completing the picture in the case of paths: for $t\in \poi$, there is a $P_t\is$ graph if and only if $t\notin \{1,4\}$. 

\begin{theorem}[Dvo\v{r}\'{a}k \cite{ptdvorak}]\label{thm:ptdvorak}
    For every integer $t\geq 6$, the graph $VD_{t-1}$ is $P_t\is$.
\end{theorem}

Naturally, cycles are the next graphs to be examined. For every integer $t\geq 3$, we denote by $C_t$ the $t$-vertex cycle, also called the \textit{$t$-cycle}. Note that there is no $C_3\is$ graph. Behrens et al. \cite{behrens} gave a simple solution for all odd cycles on five or more vertices:

\begin{theorem}[Behrens, Erbes, Santana, Yager, and Yeager \cite{behrens}]\label{thm:odd}
    For every integer $t\geq 3$, the line graph of the complete bipartite graph $K_{t,t}$ is $C_{2t-1}\is$.
\end{theorem}

In contrast, finding an $H\is$ graph for an even cycle $H$ is surprisingly difficult. We were able to find slick examples when $H$ is the $4$-cycle (also pointed out in \cite{behrens}), the $6$-cycle and the $8$-cycle (see Figure~\ref{fig:smallt}).

\begin{observation}\label{obs:c6}
    The icosahedron is $C_4\is$, the Cartesian product of two $5$-cycles is $C_6\is$, and the dodecahedron is $C_8\is$.
\end{observation}

\begin{figure}
    \centering
    \includegraphics[scale=0.55]{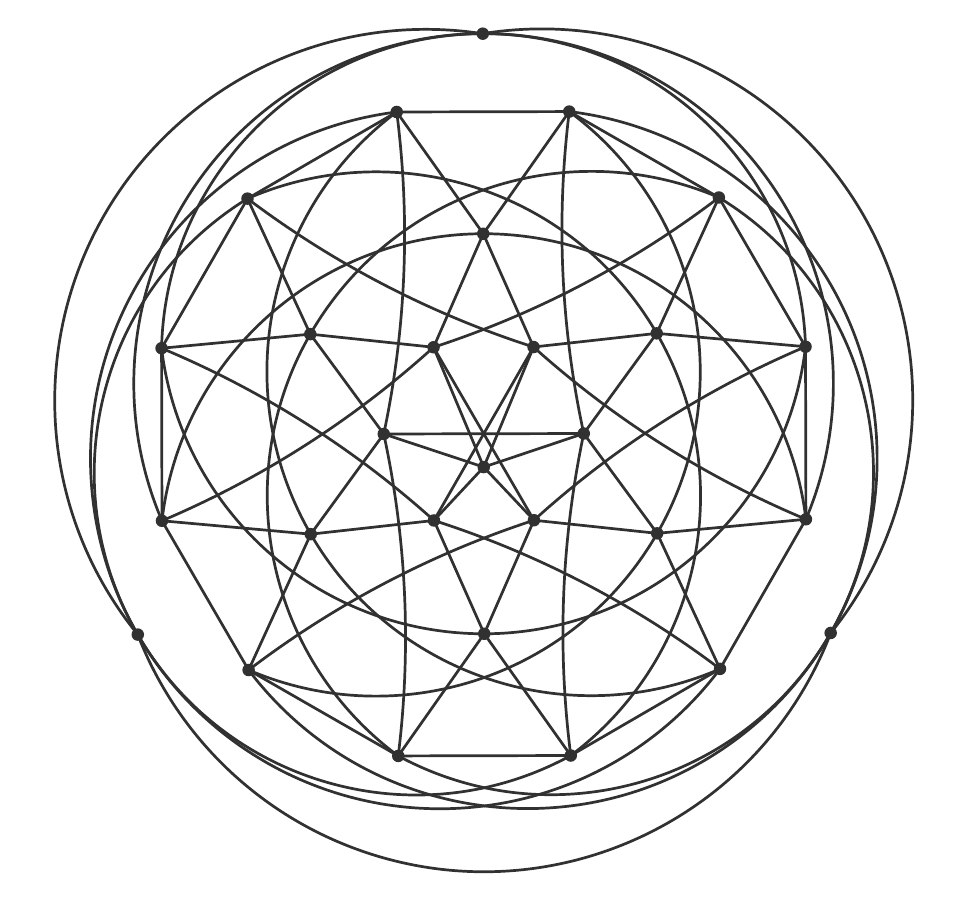}
    \caption{A $C_{10}\is$ graph.}
    \label{fig:c10is}
\end{figure}

Using a computer, we also verified that the 28-vertex graph depicted in Figure~\ref{fig:c10is} is $C_{10}\is$\footnote{We do not know if the graph in Figure~\ref{fig:c10is} is (some variant of) a well-known graph. But it exhibits rather interesting symmetries. The most curious, to us, is that it is isomorphic to the intersection graph of its own triangles.}. But that is the end of the list: it remains wide open whether induced saturated graphs exist for even cycles on $12$ or more vertices. Here is a weaker question that may be more approachable: Given an even cycle $H$ (on $12$ or more vertices), does there exist a graph $G$ that is $H$-free with at least one edge, and yet removing any edge from $G$ creates an induced copy of $H$? Our main result, Theorem~\ref{thm:maineven} below, answers this question in the affirmative (note that we only need \ref{thm:maineven} for $t\geq 7$; however, the proof also works for $t\in \{5,6\}$, so we keep that).

\begin{theorem}\label{thm:maineven}
    For every integer $t\geq 5$, there is a graph $G_t$ with $E(G_t)\neq \varnothing$ such that $G_t$ is $C_{2t-2}$-free but $G_t-e$ has an induced subgraph isomorphic to $C_{2t-2}$ for every $e\in E(G_t)$.
\end{theorem}

Two remarks: First, the proof of Theorem~\ref{thm:maineven} is already much longer and more technical than the proof of any of the results mentioned in this introduction. We hope, and it would be of great interest if, our methods can be extended to solve the induced saturation problem for all (or at least some long) even cycles. To that end, a good start would be to answer the following:

\begin{question}\label{q:add}
  Let $t\geq 6$ be an integer. Does there exist a graph $G$ with $E(\ol{G})\neq \varnothing$ such that $G$ is $H$-free but $G+e$ has an induced subgraph isomorphic to $C_{2t-2}$ for every $e\in E(\ol{G})$?
\end{question}

Second, the number of vertices in our construction of the graph $G_t$ from Theorem~\ref{thm:maineven} is doubly exponential in $t$. Could this be improved? We do not even know a rationale for expecting the number of vertices of $G_t$ to be more than linear in $t$. Let us at least ask:

\begin{question}\label{q:poly}
  Does there exist a polynomial $f$ such that for every $t\geq 3$, there is a graph $G$ on at most $f(t)$ vertices with $E(G)\neq \varnothing$, such that $G$ is $H$-free but $G-e$ has an induced subgraph isomorphic to $C_{2t-2}$ for every $e\in E(G)$?
\end{question}

This paper is organized as follows. First, to avoid repetition, we will assume for the entire remainder of the paper that $t\geq 5$ is a fixed positive integer and will prove Theorem~\ref{thm:maineven} for this fixed choice of $t$. In Section~\ref{sec:prelim}, we will set up some notation and terminology, and in particular will define ``territories'' and their ``expansions'' as the two notions that are central to the proof. A ``territory'' is simply a graph along with an induced cycle in the graph marked as the ``boundary.'' An ``expansion'' of a territory is another territory obtained by enlarging the boundary carefully so that its length increases by a (controlled) linear function of $t$.

In Section~\ref{sec:canonical}, we introduce a special kind of territories called ``canonical'' which, roughly, are built by successively gluing $t$-cycles along the boundary, and then taking expansions. We will then show that there are canonical territories with the length of the boundary equal to any sufficiently large even number, and establish several properties for canonical territories that will be used in the rest of the proof.

In Section~\ref{sec:mainconst}, we construct the graph $G_t$ that is planned to satisfy Theorem~\ref{thm:maineven}. Essentially, $G_t$ is obtained from piecing together canonical territories with long enough boundaries in such a way that they form a tiling of a (hypothetical) surface of (very) large genus; see Figure~\ref{fig:torushex}. In Section~\ref{sec:c2tfree}, we prove that $G_t$ is $C_{2t-2}$-free, and in Section~\ref{sec:critical}, we will prove that removing any edge from $G_t$ creates an induced $(2t-2)$-cycle, hence completing the proof of Theorem~\ref{thm:maineven}.
\begin{figure}[t!]
    \centering
    \includegraphics[scale=0.6]{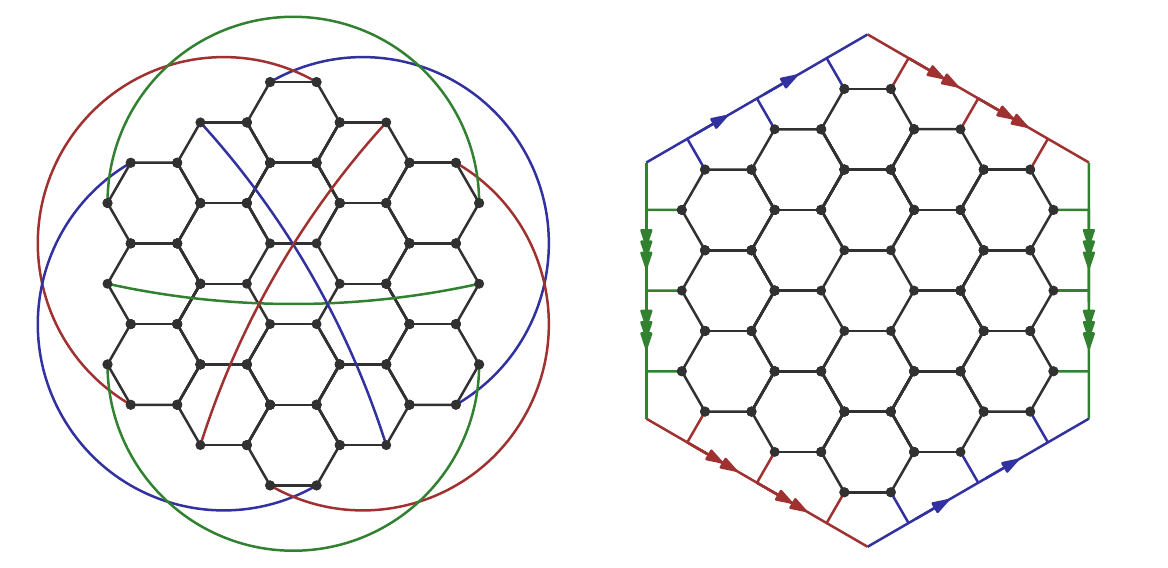}
    \caption{A $C_{10}$-free graph $G$ in which removing each edge creates an induced $C_{10}$ (left) and a drawing of $G$ on the torus that constitutes a hexagonal tiling of the torus. This graph is \underline{not} $C_{10}\is$.}
    \label{fig:torushex}
\end{figure}

\section{Territories, expansions, and other preliminaries}\label{sec:prelim}
For two graphs $G$ and $G'$, we denote by $G\cup G'$ the graph with vertex set $V(G)\cup V(G')$ and edge set $E(G)\cup E(G')$. The union of any finite number of graphs is defined analogously. Let $G=(V(G),E(G))$ be a graph. For $X\subseteq V(G)$, we denote by $G[X]$ the induced subgraph of $G$ with vertex set $X$ and by $G\setminus X$ the induced subgraph of $G$ with vertex set $V(G)\setminus X$. Two subsets $X, Y$ of $V(G)$ are \textit{anticomplete in $G$} if $X\cap Y=\varnothing$ and there is no edge of $G$ with an end in $X$ and an end in $Y$. For $x\in V(G)$ and $Y\subseteq V(G)$, we say that \textit{$x$ is anticomplete to $Y$ in $G$} to mean $\{x\}$ and $Y$ are anticomplete in $G$.

Let $P$ be a path. We write $P=p_1\dd \cdots \dd p_k$, for $k\in \poi$, to mean $V(P)=\{p_1,\ldots, p_k\}$ and $E(P)=\{p_ip_{i+1}:i\in \poi_{k-1}\}$. The vertices $p_1,p_k$ are the \textit{ends} of $P$, and $V(P)\setminus \{p_1,p_k\}$ is the \textit{interior} of $P$. The \textit{length} of $P$ is the number of edges in $P$. For $x,y\in V(P)$, we write $x\dd P\dd y$ to denote the subpath of $P$ with ends $x,y$. Similarly, for a cycle $C$, we write $P=c_1\dd \cdots \dd c_k\dd c_1$, for integer $k\geq 3$, to mean $V(C)=\{c_1,\ldots, c_k\}$ and $E(C)=\{c_ic_{i+1}:i\in \poi_{k-1}\}\cup \{c_1c_k\}$. The \textit{length} of $C$ is the number of edges in $C$. Given a graph $G$, an (induced) \textit{path in $G$} is an (induced) subgraph of $G$ that is a path, and an (induced) \textit{cycle in $G$} is an (induced) subgraph of $G$ that is a cycle. For integer $k\geq 3$, an (induced) \textit{$k$-cycle in $G$} is (induced) subgraph of $G$ that is a $k$-cycle.

A \textit{territory} is a pair $(T,B)$ where $T$ is a graph and $B$ is an induced cycle in $T$. The \textit{perimeter} of $(T,B)$ is the length of $B$. Given a territory $(T',B')$, by an \textit{expansion of $(T',B')$} we mean a territory $(T,B)$ obtained from $(T',B')$ through the following process (see Figure~\ref{fig:expansion}).  Choose a stable set $\{x_i:i\in \poi_k\}$ of cardinality $k$ in $B'$ for some $k\in \poi\cup \{0\}$ (so the stable set may be empty). For each $i\in \poi_k$, let $x^-_{i}$ and $x^+_{i}$ be the two neighbors of $x_i$ in $B'$; it follows that $\{x^-_{1},x^+_{1},\ldots,x^-_{k},x^+_{k}\}\cap \{x_1\ldots, x_k\}=\varnothing$. Choose a subset $I$ of $\poi_{k}$. Let $T$ be the graph constructed as follows:
\begin{itemize}
    \item First, for each $i\in \poi_k$, add to $T'$ a path $P_i$ of length $2t-6$ with ends $y^-_i,y^+_i$ and middle vertex $y_i$, and add the edges $x^-_iy^-_i$, $x_iy_i$ and $x^+_iy^+_i$.
    \item Then, for each $i\in I$, add to the graph constructed in the previous step a path $Q_i$ of length $t-4$ with ends $z^-_i$ and $z^+_i$, and assuming $v^-_i$ and $v^+_i$ to be the two neighbors of $y_i$ in $P_i$, add the edges $v^-_iz^-_i$ and $v^+_iz^+_i$.
\end{itemize}

This concludes the construction of the graph $T$. In order to define $B$, for each $i\in \poi_k$, let $R_i$ be the induced path in $T$ with ends $x^-_i$ and $x^+_i$, specified as follows:
\begin{itemize}
    \item If $i\in \poi_k\setminus I$, then let $R_i=x^-_i\dd y^-_i\dd P_i\dd y^+_i\dd x^+_i$; and
    \item If $i\in I$, then let $R_i=x^-_i\dd y^-_i\dd P_i\dd v^-_i\dd z^-_i\dd Q_i\dd z^+_i\dd v^+_i\dd P_i\dd y^+_i\dd x^+_i$.
\end{itemize}

Now, let 
$$B=(B\setminus \{x_1,\ldots, x_k\})\cup \left(\bigcup_{i=1}^kR_i\right).$$
Then $B$ is an induced cycle in $T$, and so $(T,B)$ is a territory.
\begin{figure}[t!]
    \centering
    \includegraphics[scale=0.7]{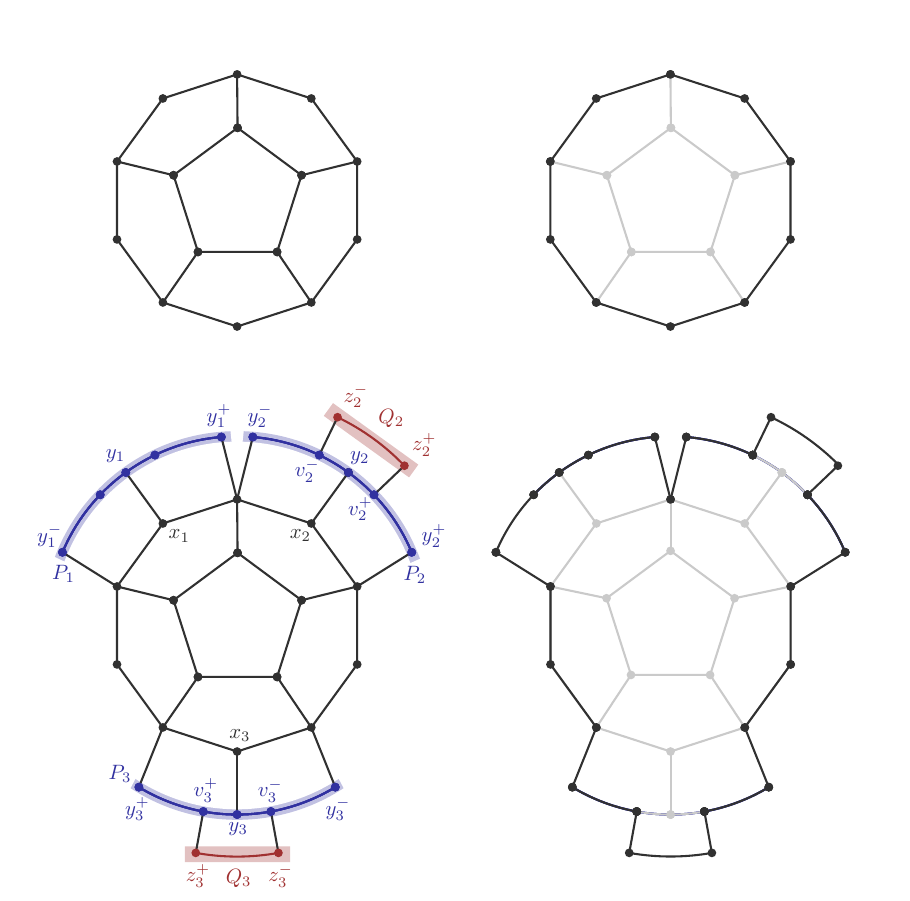}
    \caption{Top: A territory $(T',B')$ of perimeter 10 with the graph $T'$ on the left and the induced cycle $B'$ in $T'$ highlighted on the right. Bottom: An expansion $(T,B)$ of $(T',B')$ (for $t=5$) where $k=3$ and $I=\{2,3\}$, with the graph $T$ on the left and the induced cycle $B'$ in $T'$ highlighted on the right.}
    \label{fig:expansion}
\end{figure}

For instance, if $k=0$, then $T=T'$ and $B=B'$. Moreover, if $(T',B')$ has perimeter $\lambda'\in \poi$,  then  $(T,B)$ has perimeter $\lambda'+(2t-6)(k-|I|)+(3t-10)|I|$, and $2k\leq \lambda'$. In general, it is easy to see that:

\begin{observation}\label{obs:perimeter}
   Let $\lambda,\lambda'\in \poi$ and let $(T',B')$ be a territory of perimeter $\lambda'$. Then $(T',B')$ has an expansion of perimeter $\lambda$ if and only if  there are $s_1,s_2\in \poi\cup \{0\}$ for which we have
  $(2t-6)s_1+(3t-10)s_2=\lambda-\lambda'$
  and 
  $2(s_1+s_2)\leq \lambda'$.
\end{observation}

\section{Canonical territories: existence and basic properties}\label{sec:canonical}
 
 From now on, we will mostly focus on the following example of territories. For every $m\in \poi\cup \{0\}$, let $(T_m,B_m)$ be the territory of perimeter $t(t-3)^m$ defined recursively, as follows. Let $T_0$ be a $t$-cycle and let $B_0=T_0$. Assume that for some $m\in \poi$, the territory $(T_{m-1},B_{m-1})$ of perimeter $t(t-3)^{m-1}$ is defined. Then $B_{m-1}$ is a $t(t-3)^{m-1}$-cycle; say $B_{m-1}=x_1\dd\cdots\dd x_{t(t-3)^{m-1}}\dd x_1$. Now, let $T_m$ be the graph constructed by first adding to $T_{m-1}$ the $(t-4)$-subdivision $B_{m}$ of a $t(t-3)^{m-1}$-cycle $x'_1\dd\cdots \dd x'_{t(t-3)^{m-1}}\dd x'_1$, and then adding the edge $x_ix'_i$ for every $i\in \poi_{t(t-3)^{m-1}}$ (recall that the \textit{$r$-subdivision} of a graph $H$, for $r\in \poi\cup \{0\}$, is the graph obtained from $H$ by replacing each edge of $H$ with a path of length $r+1$). It follows that $B_m$ is an induced $t(t-3)^m$-cycle in $T_m$, and so $(T_m,B_m)$ is a territory of perimeter $t(t-3)^m$ (see Figure~\ref{fig:territory}). In particular, the territories $((T_m, B_m):m\in \poi)$ have arbitrarily large even perimeters. But we need more: for \textit{every} sufficiently large even number $\lambda\in \poi$, we need a territory of perimeter $\lambda$ that maintains the useful properties of $((T_m, B_m):m\in \poi)$, and that leads us to the following definition. We say that a territory $(T,B)$ is \textit{canonical} if for some $m\in \poi\cup \{0\}$, there is an expansion $(T',B')$ of $(T_m, B_m)$ as well as an isomorphism $f:V(T)\rightarrow V(T')$ between $T$ and $T'$ such that $f|_{V(B)}$ is an isomorphism between $B$ and $B'$.
 \begin{figure}
    \centering
    \includegraphics[scale=0.5]{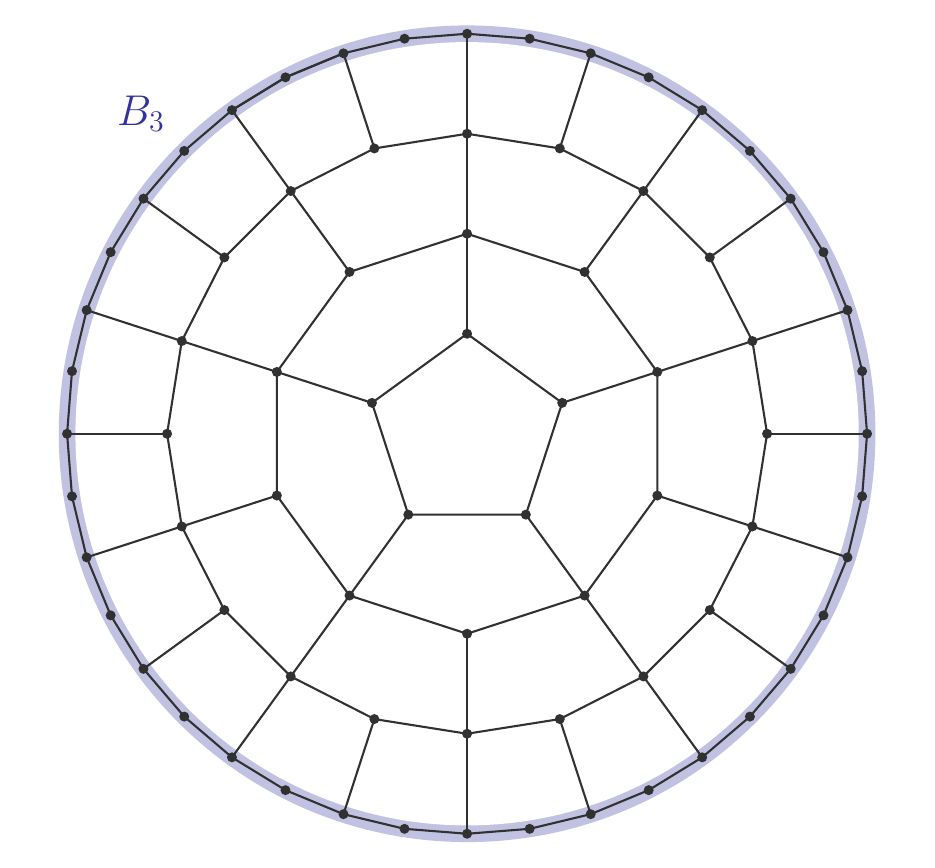}
    \caption{The territory $(T_3,B_3)$ of perimeter $t(t-3)^3$ (for $t=5$).}
    \label{fig:territory}
\end{figure}
 
 Canonical territories are the main building blocks in our construction of the graph $G_t$ that satisfies Theorem~\ref{thm:maineven}. We begin by showing that every sufficiently large even number is the perimeter of some canonical territory:

\begin{theorem}\label{thm:expansionexists}
For every even integer $\lambda\geq t^3$, there is a canonical territory of perimeter $\lambda$.
\end{theorem}

\begin{proof}
First, we show that:

\sta{\label{st:boundingthediff} There exists $m\in \poi$ such that $(2t-6)(3t-10)\leq \lambda-t(t-3)^m<t(t-3)^{m+1}$.}

Since $\lambda\geq t^3>t(t+3)(t-3)>(t^2+3t-20)(t-3)$, it follows that
$$\lambda-(2t-6)(3t-10)>(t^2+3t-20-6t+20)(t-3)=t(t-3)^2.$$
Since $(t(t-3)^m:m\in \poi)$ is a strictly increasing sequence, it follows that there is an integer $m\geq 2$ for which 
$$t(t-3)^m\leq \lambda-(2t-6)(3t-10)<t(t-3)^{m+1}.$$
In particular, we have $(2t-6)(3t-10)\leq \lambda-t(t-3)^m$. Moreover, since $m\geq 2$ and $t\geq 5$, it follows that 
$$t(t-3)^m-(2t-6)(3t-10)\geq t(t-3)^2-(2t-6)(3t-10)=(t-3)(t-4)(t-5)\geq 0.$$
Therefore,
$$\lambda-t(t-3)^m<t(t-3)^{m+1}-t(t-3)^m+(2t-6)(3t-10)\leq t(t-3)^{m+1}.$$
This proves \eqref{st:boundingthediff}.
\medskip

Henceforth, let $m\in \poi$ be as given by \eqref{st:boundingthediff}. Since $t\geq 5$, it follows that $0<2t-6<3t-10$ and so by \eqref{st:boundingthediff}, we have $\lambda - t(t-3)^m \geq (2t-6)(3t-10)>0$. Also, $\lambda-t(t-3)^m$ is even because both $\lambda$ and $t(t-3)^m$ are even.  Let $d\in \poi$ such that $\lambda - t(t-3)^m =2d$. Let $q,r$ be the (unique) integers for which $2d=(2t-6)q+r$ and $0\leq r< 2t-6$.
\medskip

Let  $s_1 = 3d-(3t-10)(q+1)$ and let $s_2 = (2t-6)(q+1)-2d$. Then we have:
\begin{align*}
(2t-6)s_1+(3t-10)s_2 &=(2t-6)(3d-(3t-10)(q+1)) +(3t-10)((2t-6)(q+1)-2d)\\
        &= (3(2t-6)-2(3t-10))d\\
        &= 2d.
\end{align*}

It follows that:

\sta{\label{st:linear_comb} 
We have $(2t-6)s_1+(3t-10)s_2 = \lambda -t(t-3)^m$.}

We further claim that:

\sta{\label{st:positive_coeff} We have $s_1,s_2\geq 0$ and $2(s_1+s_2)<t(t-3)^m$.}

By \eqref{st:linear_comb}, we have $(2t-6)s_1+(3t-10)(s_2-(2t-6)) = \lambda -t(t-3)^m-(2t-6)(3t-10)$. Thus, by \eqref{st:boundingthediff}, we have $(2t-6)s_1+(3t-10)(s_2-(2t-6))\geq 0$. Recall that $0<2t-6<3t-10$, and $s_2-(2t-6)=(2t-6)q-2d=-r\leq 0$. It follows that $s_1\geq 0$.
Also, note that $s_2=(2t-6)(q+1)-2d=2t-6-r$, and so $s_2\geq 0$ by the choice of $r$. Moreover, by \eqref{st:boundingthediff}, we have $\lambda-t(t-3)^m<t(t-3)^{m+1}$, and recall once again that $0<2t-6<3t-10$. Thus, by \eqref{st:linear_comb}, we have
$(2t-6)(s_1+s_2)<(2t-6)s_1+(3t-10)s_2 = \lambda -t(t-3)^m<t(t-3)^{m+1}$.
This proves \eqref{st:positive_coeff}.
\medskip

Now, from \eqref{st:linear_comb}, \eqref{st:positive_coeff} and Observation~\ref{obs:perimeter}, it follows that $(T_m, B_m)$ has an expansion of perimeter $\lambda$. This completes the proof of Theorem~\ref{thm:expansionexists}.
\end{proof}

Of course, we also want canonical territories to be $C_{2t-2}$-free:

\begin{theorem}\label{thm:expansionevenfree}
Every canonical territory is $C_{2t-2}$-free.
\end{theorem}

\begin{proof}
Let us first prove that:

\sta{\label{st:evenfree} For all $m\in \poi\cup \{0\}$, the graph $T_m$ is $C_{2t-2}$-free.}

We proceed by induction on $m$. The case $m=0$ is trivial. Assume that for some $m\geq 1$, there is an induced $(2t-2)$-cycle $C$ in $T_m$. Since $m\geq 1$ and $t\geq 5$, it follows that $B_m$ has length $t(t-3)^m\geq 2t$, and so $V(C) \cap V(T_{m-1}) \neq \varnothing$. Also, by the inductive hypothesis, $T_{m-1}$ is $C_{2t-2}$-free, and so $V(C) \cap V(B_m) \neq \varnothing$. Let $P$ be a component of $C[V(C) \cap V(B_m)]$. Then $P$ is an induced path in $B_m$. Let $x$ and $y$ be the ends of $P$. Then $x$ has a neighbor $x'$ in $C$ such that $x' \in V(T_{m-1})$ and $y$ has a neighbor $y'$ in $C$ such that $y' \in V(T_{m-1})$. By the construction of $(T_m, B_m)$ and since $C$ is a $(2t-2)$-cycle, it follows that $x',y'\in V(B_{m-1})$, and the length of $P$ is at most $2t-5$, and divisible by $t-3$. Since $t\geq 5$, it follows that $P$ has length $t-3$ or $2t-6$. In the former case, $x'$ and $y'$ are adjacent in $T_{m-1}$ (and so in $C$), which in turn implies that $C[V(P)\cup \{x',y'\}]$ is a $t$-cycle, a contradiction. In the latter case, again by the construction of $(T_m, B_m)$ and since $C$ is a $(2t-2)$-cycle, it follows that  $P'=C[V(C) \setminus V(P)]$ is a path of length two in $B_{m-1}$ with ends $x',y'$. But then the middle vertex of $P'$ is adjacent in $T_m$ (and so in $C$) to the middle vertex of $P$, again a contradiction. This proves \eqref{st:evenfree}.
\medskip

Now, let $(T,B)$ be an expansion of $(T_m,B_m)$ for some $m\in \poi\cup \{0\}$, and suppose for a contradiction that there is an induced $(2t-2)$-cycle $C$ in $T$. By construction, every component of $T[V(T) \setminus V(T_m)]$ is either a path of length $2t-6$ or contains exactly one cycle, and that cycle has length $t$. Thus, we have $V(C) \cap V(T_{m}) \neq \varnothing$. Moreover, by \eqref{st:evenfree}, $T_{m}$ is $C_{2t-2}$-free, and so $V(C) \setminus V(T_m) \neq \varnothing$. Let $P$ be a component of $C[V(C) \setminus V(T_m)]$. Then $P$ is an induced path in $T[V(T) \setminus V(T_m)]$. Let $x$ and $y$ be the ends of $P$. Then $x$ has a neighbor $x'$ in $V(C)\setminus V(P)$ such that $x' \in V(T_{m})$ and $y$ has a neighbor $y'$ in $V(C)\setminus V(P)$ such that $y' \in V(T_{m})$. By the construction of $(T, B)$, we have $x',y'\in V(B_{m})$, the vertices $x',y'$ are distinct and $P$ has length $t-3$, $2t-6$ or $3t-10$. Since $t \geq 5$ and $C$ is a $(2t-2)$-cycle, it follows that $P$ has length $t-3$ or $2t-6$. In the former case, $x'$ and $y'$ are adjacent in $T_{m}$ (and so in $C$), which in turn implies that $C[V(P)\cup \{x',y'\}]$ is a $t$-cycle, a contradiction. In the latter case, $C[V(C) \setminus V(P)]$ is an induced path $P'$ of length two in $B_{m}$ with end $x',y'$. But then the middle vertex of $P'$ is adjacent in $T$ (and so in $C$) to the middle vertex of $P$, again a contradiction. This completes the proof of Theorem~\ref{thm:expansionevenfree}.
\end{proof}

Given a connected graph $G$ and $x,y\in V(G)$, the \textit{distance} in $G$ between $x$ and $y$, denoted $\dist_G(x,y)$, is the length of the shortest path in $G$ between $x, y$. Here is another useful property of a canonical territory $(T, B)$: if two vertices in $V(B)$ are far apart in the cycle $B$, then they are also relatively far apart in the whole graph $T$. More precisely,
    
\begin{theorem}\label{thm:expansiondistance}
Let $d\in \poi$, let $(T, B)$ be a canonical territory and let $x,y\in V(B)$ such that $\dist_T(x,y)\leq d$. Then $\dist_B(x,y)\leq t^{d+3t}$.
\end{theorem}

\begin{proof}
We begin with the following:

\sta{\label{st:distanceunexpanded} 
Let $d\in \poi$ and let $m\in \poi\cup \{0\}$. Let $x,y\in V(B_m)$ such that $\dist_{T_m}(x,y)\leq d$. Then $\dist_{B_m}(x,y)\leq (t-2)^d$.}

The proof is by induction on $d$ (for all $m\in \poi \cup \{0\}$). The case $d=1$ is trivial. Assume that $d\geq 2$. Let $P$ be a path of length at most $d$ in $T_m$ from $x$ to $y$. Let $u$ be the unique neighbor of $x$ in $P$; thus, $u$ and $y$ are distinct. Assume that $u\in V(B_m)$. Then $u\dd P\dd y$ is a path of length at most $d-1$ in $T_m$ between $u,y\in V(B_m)$. It follows from the inductive hypothesis that $\dist_{B_m}(u,y)\leq (t-2)^{d-1}$. But then 
$$\dist_{B_m}(x,y)\leq (t-2)^{d-1}+1<(t-2)^{d}.$$
Therefore, we may assume that $u\in V(T_m)\setminus V(B_m)$. Since $u$ is adjacent to $x\in V(B_m)$, it follows that $u\in  V(B_{m-1})$; in particular, $u$ is the unique neighbor of $x$ in $V(B_{m-1})$ and $x$ is the unique neighbor of $u$ in $V(B_{m})$. Traversing $u\dd P\dd y$ from $u$ to $y$, let $w$ be the first vertex in $V(P)\cap V(B_m)$; note that $w$ exists because $y\in V(P)\cap V(B_m)$, and in fact $w$ and $y$ may be the same. Let $v$ be the unique neighbor of $w$ in $u\dd P\dd w$. It follows that $v\in  V(B_{m-1})$; in particular, $v$ is the unique neighbor of $w$ in $V(B_{m-1})$ and $w$ is the unique neighbor of $v$ in $V(B_{m})$. Note that $u\dd P \dd v$ is a path of length smaller than $d$ in $T_{m-1}$ between $u,v\in V(B_{m-1})$. Thus, by the inductive hypothesis, we have $\dist_{B_{m-1}}(u,v)\leq (t-2)^{d-1}$. This, along with the construction of $(T_m, B_m)$, implies that $\dist_{B_{m}}(x,w)\leq (t-3)(t-2)^{d-1}$. Moreover, $w\dd P\dd y$ is a path of length smaller than $d$ in $T_m$ between $w,y\in V(B_{m})$. Thus, by the inductive hypothesis, we have $\dist_{B_m}(w,y)\leq (t-2)^{d-1}$. But now
$$\dist_{B_m}(x,y)\leq (t-3)(t-2)^{d-1}+(t-2)^{d-1}=(t-2)^{d}.$$ This proves \eqref{st:distanceunexpanded}.
\medskip

Now, let $(T,B)$ be an expansion of $(T_m,B_m)$ for some $m\in \poi\cup \{0\}$ and let $x,y\in V(B)$ such that $\dist_T(x,y)\leq d$. Since $t\geq 5$ and by the construction of $(T,B)$, there exists $u\in V(B)\cap V(B_m)$ such that $\dist_B(x,u)\leq \frac{3}{2}t-1$ (for instance, if $x\in V(B)\cap V(B_m)$, then $u=x$ works). Similarly, there exists $v\in V(B)\cap V(B_m)$ such that $\dist_B(v,y)\leq \frac{3}{2}t-1$. 

We claim that:

\sta{\label{st:comparedistance}
For every shortest path $P$ in $T$ from $u$ to $v$, we have $V(P)\subseteq V(T_m)$. Consequently, $\dist_{T}(u,v)=\dist_{T_m}(u,v)$.}

Suppose for a contradiction that there is a shortest (induced) path $P$ in $T$ from $u$ to $v$ such that $V(P)\setminus V(T_m)\neq \varnothing$. Let $P'$ be a component of $P[V(P)\setminus V(T_m)]$. Then $P'$ is an induced path in some component of $T[V(T) \setminus V(T_m)]$. It follows that $u,v\notin V(P')$ because $u,v\in V(B_m)\subseteq V(T_m)$. Let $u'$ and $v'$ be the ends of $P'$, let $u''$ be the unique neighbor of $u'$ in $V(P)\setminus V(P')$ and let $v''$ be the unique neighbor of $v'$ in $V(P)\setminus V(P')$; in particular, $u'',v''$ are not adjacent in $T$. From the construction of $(T, B)$, it follows that $u'',v''\in V(B_{m})$, there is a path $u''\dd z\dd v''$ of length two in $B_m$ from $u''$ to $v''$, and $P'$ has length $2t-6$ or $3t-10$. Since $t \geq 5$, it follows that $|V(P')|\geq 5$. Consequently, $Q=T[(V(P)\setminus V(P'))\cup \{z\}]$ is a connected induced subgraph of $T$ with $x,y\in V(Q)$ and $|V(Q)|=|V(P)|-|V(P')|+1\leq |V(P)|-4$.
But now there is a path from $x$ to $y$ in $Q$, and so in $T$, which is strictly shorter than $P$, a contradiction. This proves \eqref{st:comparedistance}.
\medskip

From the choice of $u,v\in V(B)\cap V(B_m)$ and the assumption that $\dist_T(x,y)\leq d$, it follows that $\dist_T(u,v)\leq d+3t-2$ in $T$. This, along with \eqref{st:comparedistance}, implies that $\dist_{T_m}(u,v)\leq d+3t-2$. Thus, by \eqref{st:distanceunexpanded}, we have $\dist_{B_m}(u,v)\leq(t-2)^{d+3t-2}$. Since $(T,B)$ is an expansion of $(T_m,B_m)$, it follows from the construction that $\dist_B(u,v)\leq 3t(t-2)^{d+3t-2}$. But now since $\dist_B(x,u)<3t/2$ and $\dist_B(v,y)<3t/2$, it follows
$$\dist_B(x,y)<3t+3t(t-2)^{d+3t-2}=3t(1+(t-2)^{d+3t-2})\leq 3t\cdot t^{d+3t-3}<t^{d+3t}.$$
This completes the proof of Theorem~\ref{thm:expansiondistance}.
\end{proof}

\section{The main construction}\label{sec:mainconst}

Our construction of the graph $G_t$ starts with a $3$-regular, $3$-edge-colorable graph of large girth. There are various ways to build such graphs. For instance, there are bipartite $3$-regular graphs of arbitrarily large girth \cite{regbipgirth}, which are then $3$-edge-colorable by Galvin's theorem \cite{galvin}. Another example is the following result from \cite{highgirth}:

\begin{theorem}[Linial and Simkin \cite{highgirth}]\label{thm:highgirth}
    For every integer $g\geq 3$ and every even integer $n\geq 2^{2g}$, there is an $n$-vertex $3$-regular graph of girth at least $g$ with a Hamiltonian cycle.
\end{theorem}

Let 
$$g=t^{5t}.$$
By Theorem~\ref{thm:highgirth}, there is an $3$-regular graph $\Gamma$ of girth at least $g$ such that there is a Hamiltonian cycle $\Omega$ in $\Gamma$. For the remainder of the paper, let $g$ and $\Gamma$ be fixed. Let $(M_1,M_2)$ be a partition of $E(\Omega)$ into two perfect matchings of $\Gamma$, and let $M_3=E(\Gamma)\setminus E(\Omega)$. Then $(M_1,M_2,M_3)$ is a partition of $E(\Gamma)$ into three perfect matchings. 

For each $i\in \{1,2,3\}$, let $\Gamma_i$ be the spanning subgraph of $\Gamma$ with edge set $E(\Gamma)\setminus M_i$ and let $\mca{K}_i$ be the set of all components of $\Gamma_i$. Let $$\mca{K}=\mca{K}_1\cup \mca{K}_2\cup \mca{K}_3.$$ Then every element of $\mca{K}$ is an even cycle in $\Gamma$ of length at least $g$ (in fact, we have $\mca{K}_3=\{\Omega\}$). Moreover, it is readily seen that:

\begin{observation}\label{obs:cycles}
For every $e\in E(\Gamma)$, there are exactly two cycles in $\mca{K}$ containing $e$. Specifically, if $e\in M_i$ for $i\in\{1,2,3\}$, then assuming $\{i',i''\}=\{1,2,3\}\setminus \{i\}$, one of those two cycles belongs to $\mca{K}_{i'}$ and the other one belongs to $\mca{K}_{i''}$.
\end{observation}

We will often use the following lemma, which describes how short subpaths of distinct cycles in $\mca{K}$ interact:

\begin{lemma}\label{lem:twocycles}
Let $K,K'\in \mca{K}$ be distinct. Let $L$ be a path of length smaller than $t^{5t-1}$ in $K$ and let $L'$ be a path of length smaller than $t^{5t-1}$ in $K'$ such that $V(L)\cap V(L')\neq \varnothing$. Then either
\begin{itemize}
\item $V(L)\cap V(L')=\{u\}$ where $u$ is an end of at least one of $L,L'$ and $V(L)\setminus \{u\}$ and $V(L')\setminus \{u\}$ are anticomplete in $\Gamma$; or
\item $V(L)\cap V(L')=\{u,v\}$ where $u,v$ are distinct and adjacent in both $L$ and $L'$, and $L\setminus \{u,v\}$ and $L'\setminus \{u,v\}$ are anticomplete in $\Gamma$.
\end{itemize}
\end{lemma}
\begin{proof}
Let $x,y$ be the ends of $L$ and let $x',y'$ be the ends of $L'$. Without loss of generality, we may assume that $K \in \mca{K}_1$ and $K' \in \mca{K}_2$. Let $H=\Gamma[V(L)\cup V(L')]$. Since $V(L)\cap V(L')\neq \varnothing$ and since $\Gamma$ is $3$-regular with girth at least $g=t^{5t}>2t^{5t-1}>|V(H)|$, it follows that $H$ is a tree of maximum degree at most three. In particular, since $L\cup L'$ is a connected spanning subgraph of $H$, we deduce that $H=L\cup L'$. Traversing $L$ from one end to another, let $u$ and $v$ be the first and the last vertex of $L$ that belong to $V(L')$, as well. Since $H=L\cup L'$ is a tree, it follows that $Q=u\dd L\dd v=u\dd L'\dd v$ is the unique path in $H$ from $u$ to $v$, and $V(L)\setminus V(Q)$ and $V(L')\setminus V(Q)$ are anticomplete in $H$ (and so in $\Gamma$). If $u=v$, then the first bullet of \ref{lem:twocycles} holds (note that $u$ is an end of at least one of $L, L'$ because $u$ has degree at most three in $H$). So assume that $u\neq v$. Then we have $E(Q)\subseteq E(L)\cap E(L')\subseteq E(K)\cap E(K')\subseteq M_3$. Since $Q$ is a path of non-zero length in $\Gamma$ and $M_3$ is a matching in $\Gamma$, it follows that $|E(Q)|=1$; that is, $u$ and $v$ are adjacent in both $L$ and $L'$. But now the second bullet of \ref{lem:twocycles} holds. This completes the proof of Lemma~\ref{lem:twocycles}.
\end{proof}

We now construct our graph $G_t$. Start with the graph $\Gamma$. For each $K\in \mca{K}$, let $\lambda_K$ be the length of $K$. Then $\lambda_K$ is even and $\lambda_K\geq g>t^3$. By Theorem~\ref{thm:expansionexists}, there is a canonical territory of perimeter $\lambda_K$. Thus, there are graphs $(T_K:K\in \mca{K})$ with the following specifications:
\begin{itemize}
    \item For every $K\in \mca{K}$, $(T_K,K)$ is a canonical territory of perimeter $\lambda_K$.
    \item For every $K\in \mca{K}$, we have $V(T_K) \cap V(\Gamma) =V(K)$. 
    \item For all distinct $K,K'\in \mca{K}$, we have $V(T_K) \cap V(T_{K'}) = V(K) \cap V(K')$.
\end{itemize}

We define
$$G_t=\Gamma\cup \left(\bigcup_{K\in \mca{K}}T_K\right).$$
See Figure~\ref{fig:mainconst}.
\begin{figure}[t!]
    \centering
    \includegraphics[scale=0.7]{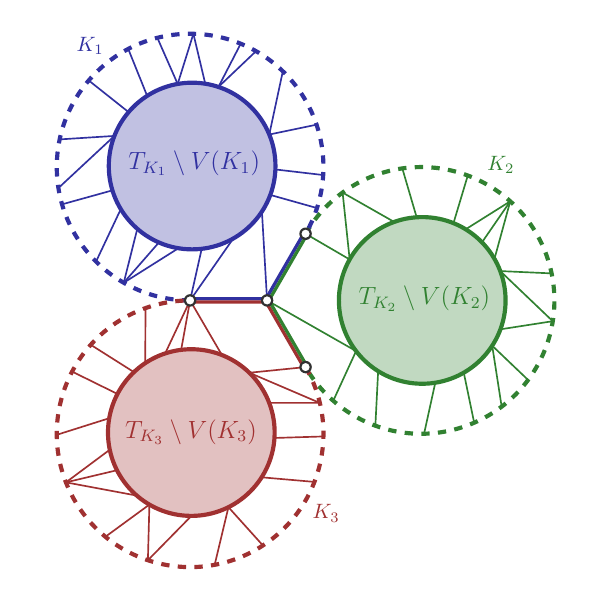}
    \caption{Construction of $G_t$: the three canonical territories $(T_{K_1},K_1), (T_{K_2},K_2)$ and $(T_{K_3},K_3)$ for three cycles $K_1,K_2,K_3\in \mca{K}$ in $\Gamma$.}
    \label{fig:mainconst}
\end{figure}

It follows that $\Gamma$ is an induced subgraph of $G_t$, and the sets $(V(T_K)\setminus V(K):K\in \mca{K})$ are pairwise anticomplete in $G_t$, and for every $K\in \mca{K}$, the sets $V(T_K)\setminus V(K)$ and $V(\Gamma)\setminus V(K)$ are anticomplete in $G_t$. In order to prove Theorem~\ref{thm:maineven}, it suffices to show that:

\begin{restatable}{theorem}{mainevenfree}\label{thm:mainevenfree}
The graph $G_t$ is $C_{2t-2}$-free.
\end{restatable}

\begin{restatable}{theorem}{mainevencrritical}
    \label{thm:mainevencrritical}
    For every $e\in E(G_t)$, there is an induced $(2t-2)$-cycle in $G_t-e$.
\end{restatable}

We will prove Theorems~\ref{thm:mainevenfree} and \ref{thm:mainevencrritical} in the next two sections.

\section{Proof of $C_{2t-2}$-freeness}\label{sec:c2tfree}

We need to prepare for the proof of Theorem~\ref{thm:mainevenfree} with a few observations (see Figure~\ref{fig:obsdetour} for Observation~\ref{obs:expansiondetour2}).

\begin{observation}\label{obs:expansiondetour1}
    Let $(T, B)$ be a canonical territory. Let $x,y\in V(B)$ be distinct, each with at least one neighbor in $V(T)\setminus V(B)$. Then $\dist_B(x,y)\geq t-3$.
\end{observation}

\begin{observation}\label{obs:expansiondetour2}
    Let $(T, B)$ be a canonical territory. Let $x,y\in V(B)$ be distinct with a common neighbor $z\in V(T)\setminus V(B)$. Let $L$ be a shortest path in $B$ from $x$ to $y$. Then the following hold.
    \begin{enumerate}[{\rm (a)}]
    \item\label{obs:expansiondetour2_a} $L$ has length at least $t-2$.
    \item\label{obs:expansiondetour2_b} Let $x',y'\in V(L)$. Assume that there is a path $P'$ of length at least two in $T$ from $x'$ to $y'$ such that $V(P')\setminus \{x',y'\}\subseteq V(T)\setminus V(B)$ and $z$ is anticomplete to $V(P')\setminus \{x',y'\}$. Then $\{x',y'\}=\{x,y\}$.
    \end{enumerate}
\end{observation}

\begin{figure}[t!]
    \centering
    \includegraphics[scale=0.6]{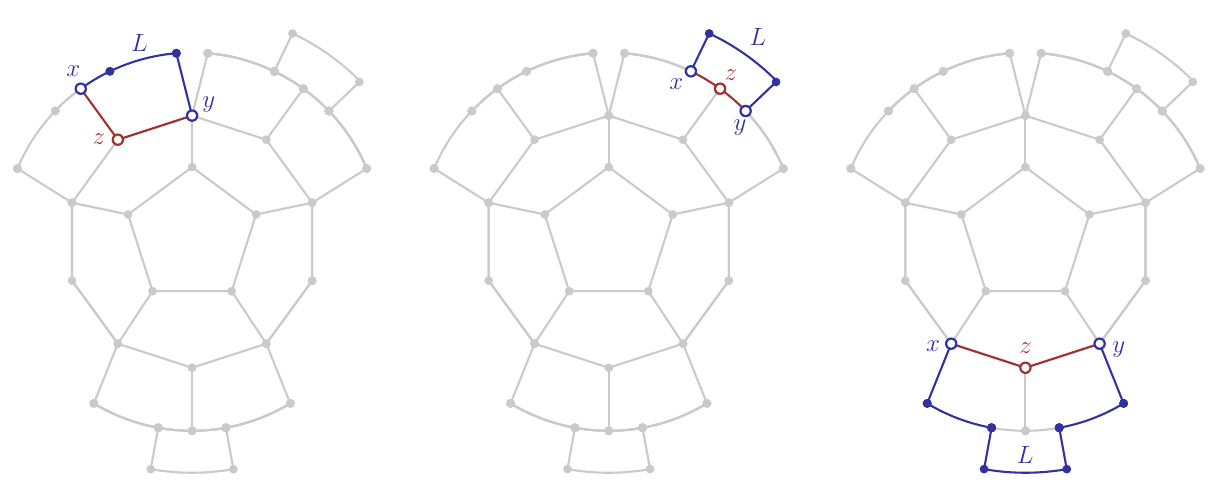}
    \caption{The possibilities for $x,y,z$ and $L$ as in Observation~\ref{obs:expansiondetour2} where $(T,B)$ is the canonical territory from Figure~\ref{fig:expansion}.}
    \label{fig:obsdetour}
\end{figure}

We also need a definition. Given a graph $W$, by an \textit{excursion in $W$} we mean, for some $k\in \poi$, an ordered $k$-tuple $(P_1,\ldots, P_k)$ of paths in $W$ with the following specifications:
\begin{itemize}
    \item There are $k+1$ vertices $x_1,\ldots, x_{k+1}\in V(W)$ such that for every $i\in \poi_k$, the ends of $P_i$ are $x_i$ and $x_{i+1}$.
    \item We have $E(W)= E(P_1)\cup \cdots\cup E(P_k)$.
\end{itemize}

We say that the excursion $(P_1,\ldots, P_k)$ is \textit{closed} if $x_1=x_{k+1}$. Observe that trees are exactly the connected graphs in which every closed excursion goes through each edge at least twice:

\begin{observation}\label{obs:walk}
    Let $W$ be a connected graph. Then there is a closed excursion $(P_1,\ldots, P_k)$ in $W$ such that some edge of $W$ belongs to exactly one of $E(P_1),\ldots, E(P_k)$, if and only if $W$ contains a cycle.
\end{observation}

We are now in a position to prove Theorem~\ref{thm:mainevenfree}, which we restate:

\mainevenfree*

\begin{proof}

Suppose for a contradiction that there is an induced $(2t-2)$-cycle $C$ in $G_t$. Roughly, the strategy is to ``project'' each piece of $C$ that is in a territory onto the boundary of that territory, and then obtain a short cycle in $\Gamma$ (which would then yield a contradiction because $\Gamma$ has large girth). To make this precise, we need several definitions. For every $K\in \mca{K}$, let $E_K=E(C)\cap (E(T_K)\setminus E(K))$; then $(E_K:K\in \mca{K})$ is a partition of $E(C)\setminus E(\Gamma)$. By a \textit{$K$-sector of $C$} we mean a path $P$ of length at least two in $C$ such that the ends of $P$ are contained $V(K)$ and the interior of $P$ is contained $V(T_K)\setminus V(K)$; in particular, we have $E(P)\subseteq E_K\subseteq E(T_K)\setminus E(K)$. Since $C$ is an induced cycle, it follows that the $K$-sectors of $C$ are paths of length at least two in $T_K$ with pairwise anticomplete interiors. Let $\mca{P}_K$ be the set of all $K$-sectors of $C$. Then $(E(P):P\in \mca{P}_K)$ is a partition of $E_K$. 

Let $$\mca{P}=\bigcup_{K\in \mca{K}}\mca{P}_K.$$
It follows that $(E(P):P\in \mca{P})$ is a partition of $E(C)\setminus E(\Gamma)$.  For every $K\in \mca{K}$ and every $P\in \mca{P}_K$, fix a shortest path $L_P$ in $K$ between the ends of $P$. Let
$$F=\bigcup_{P\in \mca{P}}E(L_P).$$
Also, let $D$ be the set of all edges (in $F$) that belong to exactly one of $(E(L_P):P\in \mca{P})$.

We claim that:

\sta{\label{st:shortpath} For every $P\in \mca{P}$, the path $L_P$ has length at most $t^{5t-2}$.}

Let $K\in \mca{K}$ such that $P\in \mca{P}_K$ and let $x,y$ be the ends of $P$. Then $P$ is a path in $T_K$ whose ends belong to $V(K)$. Since $P$ is also a path in $C$ and $C$ is a $(2t-2)$-cycle, it follows that $\dist_{T_K}(x,y)\leq 2t-2$. Since $(T_K,K)$ is canonical, it follows from Theorem~\ref{thm:expansiondistance} that $\dist_{K}(x,y)\leq t^{5t-2}$. This proves \eqref{st:shortpath}.

\sta{\label{st:containment} We have $D\subseteq E(C)\cap E(\Gamma)\subseteq F$.}

Let $C'$ be the subgraph of $C$ with edge set $E(C)\cap E(\Gamma)$ and with no isolated vertices. Let $\mca{Q}$ be the set of all components of $C'$. Then every $Q\in \mca{Q}$ is an induced path of non-zero length in $C$ (and so in $\Gamma$, because both $C$ and $\Gamma$ are induced subgraphs of $G_t$), and $(E(Q):Q\in \mca{Q})$ is a partition of $E(C)\cap E(\Gamma)$. Moreover, observe that $\mca{P}\cup \mca{Q}$ is a set of pairwise internally disjoint paths of non-zero length in $G_t$ with $C=\bigcup_{R\in \mca{P}\cup \mca{Q}}R$. In other words, writing $k=|\mca{P}\cup \mca{Q}|$, it follows that $k<2t$ and there is an enumeration $R_1,\ldots, R_k$ of the elements of $\mca{P}\cup \mca{Q}$ such that $(R_1,\ldots, R_k)$ is a closed excursion in $C$ (where every edge of $C$ belongs to exactly one of $E(R_1),\ldots, E(R_k)$). For every $i\in \poi_k$, define $\tilde{R}_i$ as follows: if $R_i\in \mca{P}$, then let $\tilde{R}_i=L_{R_i}$, and if $R_i\in \mca{Q}$, then let $\tilde{R}_i=R_i$ (this is well-defined because $\mca{P}\cap\mca{Q}=\varnothing$). Note that $\tilde{R}_i$ has the same ends as $R_i$. Also, $\tilde{R}_i$ has length at most $t^{5t-2}$ (for $R_i\in \mca{P}$, this follows from \eqref{st:shortpath}, and for $R_i\in \mca{Q}$, this follows from the fact that $\tilde{R}_i=R_i$ is a path in $C$ and $C$ is a cycle of length $2t-2<t^{5t-2}$). Let $W=\bigcup_{i=1}^k\tilde{R}_i$. Then $W$ is a connected subgraph of $\Gamma$ with $|V(W)|<2t\times t^{5t-2}<t^{5t}$, and $(\tilde{R}_1,\ldots, \tilde{R}_k)$ is a closed excursion in $W$. Now, assume that there is an edge $e$ which belongs to either $D\setminus (E(C)\cap E(\Gamma))$ or $(E(C)\cap E(\Gamma))\setminus F$. In the former case, since $e\in D$, it follows that $e$ belongs to exactly one of $(E(L_P):P\in \mca{P})$, and since $e\notin E(C)\cap E(\Gamma)$, it follows that $e$ belongs to none of $(E(Q):Q\in \mca{Q})$. In the latter case, since $e\in E(C)\cap E(\Gamma)$, it follows that $e$ belongs to exactly one of $(E(Q):Q\in \mca{Q})$, and since $e\notin F$, it follows that $e$ belongs to none of $(E(L_P):P\in \mca{P})$. Therefore, in either case, $e$ belongs to exactly one of $E(\tilde{R}_1),\ldots, E(\tilde{R}_k)$. But then by Observation~\ref{obs:walk}, there is a cycle in $W$ (and so in $\Gamma$) of length at most $|V(W)|<t^{5t}$, a contradiction to the choice of $\Gamma$ with girth at least $g=
t^{5t}$. This proves \eqref{st:containment}.

\sta{\label{st:laminar} The following hold.
\begin{itemize}
\item Let $K\in \mca{K}$ and let $P,P'\in \mca{P}_K$. Then either $E(L_P)\cap E(L_{P}')=\varnothing$ or one of $E(L_P)$ and $E(L_{P'})$ is a subset of the other.
\item Let $K, K'\in \mca{K}$ be distinct, let $P\in \mca{P}_K$ and let $P'\in \mca{P}_{K'}$. Then $|E(L_P)\cap E(L_{P'})|\leq 1$.
\end{itemize}}

Note that since $(T_K, K)$ is canonical, it follows that $T_K$ admits a planar drawing in which $K$ is the boundary of the outer face. This, combined with the fact that the $K$-sectors of $C$ are pairwise internally disjoint, implies the first assertion of \eqref{st:laminar}. For the second assertion, note that by \eqref{st:shortpath}, both $L_P$ and $L_{P'}$ have length at most $t^{5t-2}<t^{5t-1}$. Therefore, by Lemma~\ref{lem:twocycles}, we have $|E(L_P)\cap E(L_{P'})|\leq 1$. This proves \eqref{st:laminar}.
\medskip

For every $K\in \mca{K}$, let $\hat{\mca{P}}_K$ be the set of all $P\in \mca{P}_K$ for which $E(L_P)$ is maximal with respect to inclusion. In other words, $\hat{\mca{P}}_K$ is the set of all $P\in \mca{P}_K$ for which there is no $P'\in \mca{P}_K\setminus \{P\}$ with $E(L_P)\subsetneq E(L_{P'})$. It follows that $\hat{\mca{P}}_K\neq \varnothing$ if and only if $\mca{P}_K\neq \varnothing$ if and only if $E_K\neq \varnothing$.

For every $K\in \mca{K}$ and every $P\in \mca{P}_K$, let
$$D_P=E(L_P)\setminus \left(\bigcup_{P'\in \mca{P}_K\setminus \{P\}}E(L_{P'})\right).$$
Let 
$$\hat{\mca{P}}=\bigcup_{K\in \mca{K}}\hat{\mca{P}}_K.$$
We claim that:

\sta{\label{st:hatproperties} Let $P\in \mca{P}$. Then $D_P\neq \varnothing$ if and only if $P\in \hat{\mca{P}}$.}

The ``only if'' implication is immediate from the definition of $\hat{\mca{P}}_K$. For the ``if'' implication, suppose for a contradiction that for some $K\in \mca{K}$, there exists $P_0\in \hat{\mca{P}}_K$ with $D_{P_0}=\varnothing$. Let $\mca{P}_1=\{P\in \mca{P}_K\setminus \{P_0\}: E(L_{P})\subseteq E(L_{P_0})\}$ and $\mca{P}_2=\{P\in \mca{P}_K\setminus \{P_0\}:E(L_{P_0})\cap E(L_{P})= \varnothing\}$. Since $P_0\in \hat{\mca{P}}_K\subseteq \mca{P}_K$, it follows that $(\mca{P}_1,\mca{P}_2)$ is a partition of $\mca{P}_K\setminus \{P_0\}$. Let $\mca{M}$ be the set of all $P\in \mca{P}_1$ for which there is no $P'\in \mca{P}_1\setminus \{P\}$ with $E(L_P)\subsetneq E(L_{P'})$. Since $P_0\in \hat{\mca{P}}_K$ and $D_{P_0}=\varnothing$, it follows from the first bullet of \eqref{st:laminar} that $(E(L_P):P\in \mca{M})$ is a partition of $E(L_{P_0})$ (and so $\mca{M}\neq \varnothing$). We deduce that $H=P_0\cup (\bigcup_{P\in \mca{M}}P)$ is a cycle in $G_t$ with $E(H)\subseteq E_K\subseteq E(C)\cap E(T_K)$. Since $E(H)\subseteq E(C)$ and $C$ is an induced $(2t-2)$-cycle in $G_t$, it follows that $H=C$. But now since $E(H)\subseteq E(T_K)$, it follows $H$ is an induced $(2t-2)$-cycle in $T_K$, contrary to Theorem~\ref{thm:expansionevenfree} as $(T_K, K)$ is canonical. This proves \eqref{st:hatproperties}.

\sta{\label{st:longcycle} The following hold.
\begin{itemize}
    \item For every $K\in \mca{K}$, the sets $(E(L_P): P\in \hat{\mca{P}}_K)$ are pairwise disjoint. 
\item For every $P\in \hat{\mca{P}}$, we have $|E(P)|+|D_P|\geq t$.
\end{itemize}}

Suppose that for some $K\in \mca{K}$, there are distinct $P,P'\in \hat{\mca{P}}_K$ for which $E(L_P)\cap E(L_{P'})\neq \varnothing$. By the first bullet of \eqref{st:laminar}, we have $E(L_P)=E(L_{P'})$. But then $D_P=D_{P'}=\varnothing$, a contrary to \eqref{st:hatproperties}. This proves the first assertion of \eqref{st:longcycle}. For the second assertion, let $P\in \hat{\mca{P}}_K$ for some $K\in \mca{K}$ and let $x,y$ be the ends of $P$; thus, $x,y$ are the ends of $L_P$, as well. By \eqref{st:hatproperties}, we have $D_P\neq \varnothing$. Choose $e\in D_P\subseteq E(L_P)$. Assume that $P$ has length at least three. By Observation~\ref{obs:expansiondetour1}, there is a path $L'$ of length $t-3$ in $L_P$ such that $e\in E(L')$ and every vertex in the interior of $L'$ has degree two in $T_K$. Since $e\in D_P$, it follows that $E(L')\subseteq D_P$. But then $|E(P)|+|D_P|\geq 3+(t-3)=t$, as desired. Now, assume that $P$ has length two; say $P=x\dd z\dd y$ where $z\in V(T_K)\setminus V(K)$. Apply Observations~\ref{obs:expansiondetour2} to $x,y,z$ and $L_P$. By \ref{obs:expansiondetour2}\ref{obs:expansiondetour2_a}, we have $|E(L_P)|\geq t-2$. We further claim that $D_P=E(L_P)$. Suppose not. Then, since $P\in \hat{\mca{P}}_K$, it follows that there exists $P'\in \mca{P}_K\setminus \{P\}$ with $E(L_{P'})\subseteq E(L_P)$. Let $x',y'\in V(K)$ be the ends of $P'$; thus, $x',y'\in V(L_P)$. Since $P,P'\in \mca{P}_K$ are distinct, it follows that $z$ is anticomplete to $P'\setminus \{x',y'\}$. Therefore, by \ref{obs:expansiondetour2}\ref{obs:expansiondetour2_b}, we have $\{x,y\}=\{x',y'\}$ and so $L_P=L_{P'}$. But then $e\in E(L_P)=E(L_P)\cap E(L_{P'})$, contrary to the choice of $e\in D_P$. The claim follows, which in turn implies that $|E(P)|+|D_P|=|E(P)|+|E(L_P)|\geq 2+(t-2)=t$. This proves \eqref{st:longcycle}.
\medskip

Define $\sigma,\rho,\kappa\in \poi\cup \{0\}$ as follows:
\begin{itemize}
    \item let $\sigma=|\hat{\mca{P}}|=\sum_{K\in \mca{K}}|\hat{\mca{P}}_K|$;
    \item let $\rho$ be the number of all $2$-subsets $\{P,P'\}$ of $\hat{\mca{P}}$ for which $D_P\cap D_{P'}\neq \varnothing$; and
    \item let $\kappa=|\{K\in \mca{K}:|\hat{\mca{P}}_K|\geq 2\}|$. 
\end{itemize}

We claim that:

\sta{\label{st:ineq2} $\displaystyle \kappa\leq \binom{\sigma}{2}-\rho$.}

By the first bullet of \eqref{st:longcycle}, the number of $2$-subsets $\{P,P'\}$ of $\hat{\mca{P}}$ with $D_P\cap D_{P'}=\varnothing$ is at least $\kappa$. It follows that $\rho\leq \binom{\sigma}{2}-\kappa$. This proves \eqref{st:ineq2}.

\sta{\label{st:ineq1} $2\rho\geq t\sigma-|E(C)\setminus E(\Gamma)|-|D|$.}

Let $U=\bigcup_{P\in \hat{\mca{P}}}D_P$. By Observation~\ref{obs:cycles} and the first bullet of \eqref{st:longcycle}, every edge in $U$ belongs to one or two of $(D_P: P\in \hat{\mca{P}})$. Let $U_1$ be the set of all edges that belong to at exactly one of $(D_P: P\in \hat{\mca{P}})$ and let $U_2$ be the set of all edges that belong to exactly two of $(D_P: P\in \hat{\mca{P}})$. Then  $\sum_{P\in \hat{\mca{P}}}|D_P|=|U_1|+2|U_2|$. By \eqref{st:hatproperties}, we have $U_1=D$ (recall that $D$ is the set of all edges that belong exactly one of $(E(L_P):P\in \mca{P})$). Also, by the first bullet of \eqref{st:longcycle} and the second bullet of \eqref{st:laminar}, for every $2$-subset $\{P,P'\}$ of $\hat{\mca{P}}$, we have $|D_P\cap D_{P'}|\leq 1$. It follows that $|U_2|=\rho$. Therefore,
$$\sum_{P\in \hat{\mca{P}}}|D_P|=|D|+2\rho.$$
On the other hand, by the second bullet of \eqref{st:longcycle}, for every $P\in \hat{\mca{P}}$, we have $|D_P|\geq t-|E(P)|$. Hence, 
$$\sum_{P\in \hat{\mca{P}}}|D_P|\geq t\sigma-\sum_{P\in \hat{\mca{P}}}|E(P)|\geq t\sigma-\sum_{P\in \mca{P}}|E(P)|=t\sigma-|E(C)\setminus E(\Gamma)|.$$
This proves \eqref{st:ineq1}.

\sta{\label{st:sigma} $2\leq \sigma\leq t-1$.}

Since $C$ has length $2t-2$ and $\Gamma$ has girth at least $g>2t-2$, we have $E(C)\setminus E(\Gamma)\neq \varnothing$, and so $E_K\neq \varnothing$ for some $K\in \mca{K}$. It follows that $\hat{\mca{P}}_K\neq \varnothing$, and so $\sigma\geq 1$. Assume that $\sigma=1$. Then there is exactly one $K\in \mca{K}$ such that $\hat{\mca{P}}_K\neq \varnothing$, and so $E_{K'}=\mca{P}_{K'}=\hat{\mca{P}}_{K'}=\varnothing$ for every $K'\in \mca{K}\setminus \{K\}$. It follow that $E(C)\setminus E(\Gamma)=E_K\subseteq E(T_K)$ and $F\subseteq E(K)\subseteq E(T_K)$. The latter combined with \eqref{st:containment} implies that $E(C)\cap E(\Gamma)\subseteq F\subseteq E(T_K)$. Therefore, we have $E(C)\subseteq E(T_K)$. But now $C$ is an induced subgraph of $T_K$, a contradiction to Theorem~\ref{thm:expansionevenfree} as $(T_K, K)$ is canonical. We deduce that $\sigma\geq 2$. Moreover, recall that the elements of $\hat{\mca{P}}\subseteq \mca{P}$ are pairwise internally disjoint paths in $C$, each of length at least two. Since $C$ has length $2t-2$, it follows that $\sigma=|\hat{\mca{P}}|\leq t-1$. This proves \eqref{st:sigma}.
\medskip

Now, we have
\begin{align*}
2\kappa &\leq \sigma^2-\sigma-2\rho\\
        &\leq \sigma^2-(t+1)\sigma+|E(C)\setminus E(\Gamma)|+|D|\\
        &\leq \sigma^2-(t+1)\sigma+|E(C)\setminus E(\Gamma)|+|E(C)\cap E(\Gamma)|\\
        &=\sigma^2-(t+1)\sigma+2t-2;
\end{align*}
where the first inequality follows from \eqref{st:ineq2}, the second inequality follows from \eqref{st:ineq1}, and the third inequality follows from \eqref{st:containment}. On the other hand, note that $\kappa\geq 0$, and by \eqref{st:sigma}, we have $\sigma^2-(t+1)\sigma+2t-2=(\sigma-2)(\sigma-(t-1))\leq 0$. It follows that $2\kappa=\sigma^2-\sigma-2\rho=\sigma^2-(t+1)\sigma+2t-2=0$ and $|D|=|E(C)\cap E(\Gamma)|$. This, along with \eqref{st:containment}, implies that:

$$\kappa=0, \quad \displaystyle \rho=\binom{\sigma}{2}, \quad \sigma\in \{2,t-1\}, \quad 
D=E(C)\cap E(\Gamma).$$

Since $\kappa=0$, it follows that there are $\sigma$ pairwise distinct cycles $K_1,\ldots, K_{\sigma}\in \mca{K}$ for which $|\hat{\mca{P}}_{K_1}|=\cdots =|\hat{\mca{P}}_{K_{\sigma}}|=1$ and $\hat{\mca{P}}=\hat{\mca{P}}_{K_1}\cup \cdots \cup \hat{\mca{P}}_{K_{\sigma}}$; in particular, we have $E_{K}=\mca{P}_{K}=\hat{\mca{P}}_{K}=\varnothing$ for all $K\in \mca{K}\setminus \{K_1,\ldots, K_{\sigma}\}$. For each $i\in \poi_{\sigma}$, let $\hat{\mca{P}}_{K_i}=\{P_i\}$, and write 
$$D_i=D_{P_i}, \quad E_i=E(P_i), \quad F_i=E(L_{P_i}).$$

Assume that $\sigma=2$. Since $\rho=\binom{\sigma}{2}$, by the second bullet of \eqref{st:laminar}, there exists an edge $e\in E(K_1)\cap E(K_2)$ for which $D_1\cap D_2=F_1\cap F_2=\{e\}$. In particular, $D=(D_1\cup D_{2})\setminus \{e\}$, and since $D=E(C)\cap E(\Gamma)$, it follows that $|E(C)\cap E(\Gamma)|=|D|=|D_1|+|D_2|-2$. Also, from the second bullet of \eqref{st:longcycle}, it follows that  $|E_1\cup E_2|=|E_1|+|E_2|\geq 2t-(|D_1|+|D_2|)$. Therefore,
$|E_1\cup E_2|\geq 2t-2-|E(C)\cap E(\Gamma)|=|E(C)\setminus E(\Gamma)|$.
Since $E_1\cup E_2\subseteq E(C)\setminus E(\Gamma)$, it follows that $E(C)\setminus E(\Gamma)=E_1\cup E_2$. In particular, we have $\mca{P}_{K_1}=\hat{\mca{P}}_{K_1}$ and $\mca{P}_{K_2}=\hat{\mca{P}}_{K_2}$, which in turn imply that $D_1=F_1$ and $D_2=F_2$. Consequently, $E(C)\cap E(\Gamma)=(F_1\cup F_2)\setminus \{e\}$, and so $E(C)=E_1\cup E_2\cup ((F_1\cup F_2)\setminus \{e\})$. But now the ends of $e$ belong to $V(C)$ whereas $e\notin E(C)$, contrary to the fact that $C$ is an induced subgraph of $G_t$.

Finally, assume that $\sigma=t-1$. Since $\rho=\binom{\sigma}{2}$, it follows that $D_1,\ldots, D_{t-1}$ are pairwise intersecting, which in turn implies that $E(K_1),\ldots, E(K_{t-1})$ are pairwise intersecting. On the other hand, since $K_1,\ldots, K_{t-1}\in \mca{K}=\mca{K}_1\cup \mca{K}_2\cup \mca{K}_3$ and $t\geq 5$, it follows that there are distinct $i,j\in \poi_{t-1}$ and $k\in \{1,2,3\}$ such that $K_i,K_j\in \mca{K}_k$. But then $E(K_i)\cap E(K_j)=\varnothing$, a contradiction. This completes the proof of Theorem~\ref{thm:mainevenfree}
.
\end{proof}

\section{Proof of edge-criticality}\label{sec:critical}

In this last section, we prove Theorem~\ref{thm:mainevencrritical}, hence completing the proof of Theorem~\ref{thm:maineven}. We will need the following:
\begin{observation}\label{obs:expansionbasic}
    Let $(T, B)$ be a canonical territory. Then the following hold.
    \begin{enumerate}[{\rm (a)}]
     \item\label{obs:expansionbasic_a} For every edge $e\in E(T)\setminus E(B)$, there is an induced $(2t-2)$-cycle in $T-e$.
        \item\label{obs:expansionbasic_b} For every $e\in E(B)$, there is an induced $t$-cycle $H$ in $T$ such that $e\in E(H)$. 
    \end{enumerate}
\end{observation}

We will also need two lemmas:

\begin{lemma}\label{lem:isgpreserved}
Let $K\in \mca{K}$, let $H$ be a connected induced subgraph of $T_K$ such that the vertices of $H$ are pairwise at distance at most $t-1$ in $H$. Then $H$ is an induced subgraph of $G_t$.
\end{lemma}
\begin{proof}
Suppose not. Then there is an edge $xy$ in $E(G_t)\setminus E(T_K)$ with $x,y\in V(H)\subseteq V(T_K)$. Let us first show that:

\sta{\label{st:chords} We have $x,y\in V(K)$, $xy\in E(\Gamma)\setminus E(K)$ and $\dist_K(x,y)\geq t^{5t-1}$.}

Let $\mca{K}' = \mca{K}\setminus\{K\}$. Since $xy\in E(G_t) \setminus E(T_K) = E(\Gamma)\cup (\bigcup_{K'\in \mca{K}'}E(T_{K'}))$, it follows that $x, y\in V(\Gamma)\cup (\bigcup_{K'\in \mca{K}'}V(T_{K'}))$. Thus, since $x, y \in V(T_K)$, it follows that:
$$x, y \in (V(\Gamma) \cap V(T_K))\cup \left( \bigcup_{K'\in \mca{K}'}(V(T_{K})\cap V(T_{K'}))\right) $$
and so
$$x, y \in V(K) \cup \left( \bigcup_{K'\in \mca{K}'}(V(K)\cap V(K'))\right)=V(K) \subseteq V(\Gamma).$$
Since $xy \in E(G_t)$ and $\Gamma$ is an induced subgraph of $G_t$, it follows that $xy \in E(\Gamma)$. Also, since $xy \notin E(T_K)$ and $K$ is an induced subgraph of $T_K$, it follows that $xy \notin E(K)$. In conclusion, we have $x,y\in V(K)$ and $xy\in E(\Gamma)\setminus E(K)$. It remains to show that $\dist_K(x,y)\geq t^{5t-1}$. Suppose not. Then there is a path $L$ of length smaller than $t^{5t-1}$ in $K$ from $x$ to $y$. Since $K$ is a subgraph of $\Gamma$ and since $xy\in E(\Gamma)\setminus E(K)$, it follows that $L$ is a path of length smaller than $t^{5t-1}$ in $\Gamma-xy$ from $x$ to $y$. But now $\Gamma[V(L)]$ contains a cycle of length at most $t^{5t-1}<t^{5t}=g$, contrary to the choice of $\Gamma$ with girth at least $g$. This proves \eqref{st:chords}.
\medskip

By \eqref{st:chords}, we have $x,y\in V(K)$ and $\dist_K(x,y)\geq t^{5t-1}$. On the other hand, since $x,y\in V(H)$ and $H$ is an induced subgraph of $T_K$, it follows that $\dist_{T_K}(x,y)\leq \dist_{H}(x,y)\leq t-1$. But now since $(T_K,K)$ is canonical, it follows from Theorem~\ref{thm:expansiondistance} that $\dist_{T_K}(x,y)\leq t^{7t-1}<t^{5t-1}$, a contradiction. This completes the proof of Lemma~\ref{lem:isgpreserved}.
\end{proof}

\begin{lemma}\label{lem:multidistance}
 Let $d\in \poi$, let $K\in \mca{K}$ and let $u\in V(T_K)$. Then there is a path $L$ in $K$ of length at most $t^{2d+3t}$ such that for every $x\in V(K)$ with $\dist_{T_K}(u,x)<d$, we have $x\in V(L)$.
\end{lemma}

\begin{proof}
    Let $D=\{x\in V(K):\dist_{T_K}(u,x)<d\}$. If $D=\varnothing$, then any path $L$ in $K$ satisfies the lemma. So we may assume that $D\neq \varnothing$. Fix a vertex $x_0\in D$. Then, for every $x\in D$, we have $\dist_{T_K}(x_0,x)\leq \dist_{T_K}(u,x_0)+\dist_{T_K}(u,x)<2d$. Since $(T_K,K)$ is canonical, it follows from Theorem~\ref{thm:expansiondistance} that for every $x\in D$, we have $\dist_K(x_0,x)\leq t^{2d+3t-1}$. Hence, there is a path $L$ in $K$ of length at most $2t^{2d+3t-1}<t^{2d+3t}$ such that $D\subseteq V(L)$. This completes the proof of Lemma~\ref{lem:multidistance}.
\end{proof}

We are now ready to prove Theorem~\ref{thm:mainevencrritical}, which we restate:

\mainevencrritical*

\begin{proof}
Assume that $e\in E(G_t)\setminus E(\Gamma)$. Then there exists $K\in \mca{K}$ such that $e\in E(T_K)\setminus E(\Gamma)$, and so $e\in E(T_K)\setminus E(K)$. Since $(T_K,K)$ is canonical, by Observation~\ref{obs:expansionbasic}\ref{obs:expansionbasic_a}, there is an induced $(2t-2)$-cycle in $T_K-e$. It follows that there is an induced subgraph $H$ of $T_K$ with $e\in E(H)$ such that $H-e$ is a $(2t-2)$-cycle. In particular, the vertices of $H$ are pairwise at distance at most $t-1$ in $H$. Therefore, by Lemma~\ref{lem:isgpreserved}, $H$ is an induced subgraph $G_t$, and so $H-e$ is an induced $(2t-2)$-cycle in $G_t-e$, as desired.

From now on, assume that $e\in E(\Gamma)$. Let $u,v$ be the ends of $e$. Without loss of generality, we may assume that $e\in M_3$, and so by Observation~\ref{obs:cycles}, there exist $K_1\in \mca{K}_1$ and $K_2\in \mca{K}_2$ such that $e\in E(K_1)\cap E(K_2)$.

Let $i\in \{1,2\}$ be fixed. Since $(T_{K_i}, K_i)$ is canonical, it follows from Observation~\ref{obs:expansionbasic}\ref{obs:expansionbasic_b} that there is an induced $t$-cycle $H_i$ in $T_{K_i}$ such that $e\in E(H_i)$. Thus, by Lemma~\ref{lem:isgpreserved}, $H_i$ is an induced subgraph of $G_t$, as well. Let $A_i=V(H_i)\cap V(K_i)$ and let $A'_i=V(H_i)\setminus A_i$. Then $u,v\in A_i$. Since $(T_{K_i},K_i)$ is canonical $u\in V(T_{K_i})$, by Lemma~\ref{lem:multidistance}, there is a path $L_i$ in $K_i$ of length at most $t^{5t-2}$ such that for every $x\in V(K_i)$ with $\dist_{T_{K_i}}(u,x)<t-1$, we have $x\in V(L_i)$. In particular, we have $A_i\subseteq V(L_i)$ because $u\in A_i\subseteq V(H_i)$ and $H_i$ is a $t$-cycle in $T_{K_i}$. 

Since $L_1,L_2$ both have length at most $t^{5t-2}<t^{5t-1}$, it follows from Lemma~\ref{lem:twocycles} that:

\sta{\label{st:antibits} $V(L_1)\setminus \{u,v\}$ and $V(L_2)\setminus \{u,v\}$ are anticomplete in $\Gamma$, and so in $G_t$.}

We further claim that:

\sta{\label{st:antibits2} $V(H_1)\setminus \{u,v\}$ and $V(H_2)\setminus \{u,v\}$ are anticomplete in $G_t$.}

From \eqref{st:antibits}, it follows that $A_1\setminus \{u,v\}\subseteq V(L_1)\setminus \{u,v\}$ and $A_2\setminus \{u,v\}\subseteq V(L_2)\setminus \{u,v\}$ are anticomplete in $G_t$. In particular, since $V(H_1)\cap V(H_2)\subseteq (V(H_1)\cap V(K_1))\cap (V(H_2)\cap V(K_2))=A_1\cap A_2$, it follows that $V(H_1)\setminus \{u,v\}$ and $V(H_2)\setminus \{u,v\}$ are disjoint. Moreover, note that $A'_1\subseteq V(T_{K_1})\setminus V(K_1)$ and $A'_2\subseteq V(T_{K_2})\setminus V(K_2)$ are anticomplete in $G_t$. So it remains to show that $A'_1$ and $A_2\setminus \{u,v\}$ are anticomplete in $G_t$ and $A'_2$ and $A_1\setminus \{u,v\}$ are anticomplete in $G_t$. Suppose not. Then we may assume without loss of generality that there is an edge $x_1x_2\in E(G_t)$ with $x_1\in A'_1\subseteq V(T_{K_1})\setminus V(K_1)$ and $x_2\in A_2\setminus \{u,v\}\subseteq V(K_2)$. Since $x_1x_2$ has an end in $V(T_{K_1})\setminus V(K_1)$, it follows that $x_1x_2\in E(T_{K_1})$, and so $x_2\in V(T_{K_1})\cap V(K_2)\subseteq V(K_1)\cap V(K_2)\subseteq V(K_1)$. Also, since $u,x_1\in V(H_1)$ and $H_1$ is a $t$-cycle in $T_{K_1}$ with $t\geq 5$, it follows that $\dist_{T_{K_1}}(u,x_1)\leq t-3$. This, along with the fact that $x_1x_2\in E(T_{K_1})$, implies that $\dist_{T_{K_1}}(u,x_2)\leq t-2$, and so $x_2\in V(L_1)\setminus \{u,v\}$. On the other hand, recall that $x_2\in A_2\setminus \{u,v\}\subseteq V(L_2)\setminus \{u,v\}$. But now $x_2\in (V(L_1)\setminus \{u,v\})\cap (V(L_2)\setminus \{u,v\})$, contrary to \eqref{st:antibits}. This proves \eqref{st:antibits2}.
\medskip

Since $H_1,H_2$ are induced $t$-cycles in $G_t$ with $e=uv\in E(H_1)\cap E(H_2)$, it follows from \eqref{st:antibits2} that $(H_1\cup H_2)-e$ is an induced $(2t-2)$-cycle in $G_t-e$. This completes the proof of Theorem~\ref{thm:mainevencrritical}.
\end{proof}

\bibliographystyle{abbrv}
\bibliography{ref}

\end{document}